\documentclass[11pt]{amsart} 
\usepackage{amssymb, graphicx}
\usepackage[latin1]{inputenc}
\usepackage{amsfonts}
\usepackage{latexsym}
\usepackage{amscd}
\usepackage{enumerate}
\newtheorem{prop}{Proposition}[section]
\newtheorem{lem}[prop]{Lemma}
\newtheorem{thm}[prop]{Theorem}
\newtheorem{cor}[prop]{Corollary}
\newtheorem{definition}[prop]{Definition}
\newtheorem{definitions}[prop]{Definitions}
\newtheorem{ejm}[prop]{Example}

\newtheorem{rem}[prop]{Remark}
\newtheorem{rems}[prop]{Remarks}

\numberwithin{equation}{section}

\newcommand{\N}{\mathbb{N}}
\newcommand{\R}{\mathbb{R}}
\newcommand{\M}{\mathbb{M}}

\newcommand{\QAn}{\mathrm{QAnn}}
\newcommand{\Ann}{\mathrm{Ann}}
\newcommand{\ann}{\mathrm{Ann}_}
\newcommand{\der}{\mathrm{Der}}
\newcommand{\dergr}{\mathrm{Der}_{gr}}
\newcommand{\pder}{\mathrm{PDer}}
\newcommand{\pdergr}{\mathrm{PDer}_\mathrm{gr}}
\newcommand{\Igre}{\mathcal{I}_{gr-e}}
\newcommand{\Ad}{\mathrm{ad}\,}

\newcommand{\End}{\mathrm{End}}
\newcommand{\Z}{\mathbb{Z}}
\newcommand{\F}{\mathcal{F}}
\newcommand{\ider}{\mathrm{IDer}}
\newcommand{\tkk}{\mathrm{TKK}}
\newcommand{\ad}{\mathrm{ad}_}
\newcommand{\Hom}{\mathrm{Hom}}

\newcommand{\Supp}{\mathrm{Supp}}

\newcommand{\limit}{\mathrm{lim}}

\begin{document}

\title{Algebras of quotients of graded Lie algebras}


\author{Juana S\'anchez Ortega, Mercedes
Siles Molina}
\address{Departamento de \'Algebra, Geometr{\'\i}a y Topolog{\'\i}a, Universidad de
M{\'a}laga, 29071 M{\'a}laga, Spain} \email{jsanchez@agt.cie.uma.es,
msilesm@uma.es}

\maketitle

\begin{abstract}
In this paper we explore graded algebras of quotients of Lie
algebras with special emphasis on the 3-graded case and
answer some natural questions concerning its relation to
 maximal Jordan systems of quotients.
\end{abstract}

\section*{Introduction}

The study of algebras of quotients of Lie algebras was iniciated
by the second author in~\cite{msm}. Inspired by the associative
definition of ring of quotients given by Utumi in~\cite{Utumi} and
adapting some ideas from
 ~\cite{mart}, a
notion of general algebra of quotients of a Lie algebra was
introduced and also, as a special concrete example,  a maximal algebra of quotients for every
semiprime Lie algebra was built.

The introductory paper ~\cite{msm} was followed by \cite{CabreraSanchez},  where abstract properties of algebras of quotients
were considered, and by \cite{bpss}, where the main objective was to
compute maximal algebras of quotients of Lie algebras of the form
$A^-/Z_A$ and $K/Z_K$, for $A$ a prime associative algebra, $K$ the
Lie algebra of skew elements of a prime associative algebra with
involution and $Z_A$ and $Z_K$ their respective centers.

In the present article  we are interested in algebras of quotients
of graded Lie algebras. One of the reasons to consider them
is the relationship between 3-graded Lie algebras and Jordan pairs: recall that given a Jordan pair
$V$, there exists a 3-graded Lie algebra, $TKK(V)$, such that
$TKK(V)_{-1}=V^-$ and $TKK(V)_{1}=V^+$ (and other additional
properties), where $TKK(V)$ stands for the Tits-Kantor-Koecher Lie
algebra associated to $V$.

Apart from a preliminary section, we have divided the paper into
three parts. We start by introducing in Section 2 graded
algebras of quotients of Lie algebras. As in the non-graded case,
graded algebras of quotients of graded Lie algebras inherit
primeness, semiprimeness and strong nondegeneracy. We study the
relationship between the graded and the non-graded notions of
quotients, and give important examples of graded algebras of
quotients of Lie algebras. Concretely, for $A=\oplus_{\sigma\in
G}A_\sigma$ a graded semiprime associative algebra with an
involution $*$ such that $A_{\sigma}^\ast=A_\sigma$, $K_A$ the Lie
algebra of all the skew elements of $A$, and $Z_{K_A}$ the
annihilator of $K_A$, we prove  that if $Q$ is a $G$-graded
$*$-subalgebra of $Q_s(A)$ (the Martindale symmetric algebra of
quotients of $A$) containing $A$, then $K_Q/Z_{K_Q}$ is a
$G$-graded Lie algebra of quotients of $K_A/Z_{K_A}$, and
$[K_Q,K_Q]/Z_{[K_Q,K_Q]}$ is a graded Lie algebra of quotients of
$[K_A, K_A]/Z_{[K_A,K_A]}$. As a corollary we obtain that
$Q^-/Z_Q$ and $[Q^-,Q^-]/Z_{[Q,Q]}$ are graded algebras of
quotients of $A^-/Z_A$ and $[A^-,A^-]/Z_{[A,A]}$, respectively.

Section 3 deals with the particular case of 3-graded Lie algebras. Finally, in Section 4 we
relate Jordan systems of quotients and Lie algebras of quotients,
where by a Jordan system we understand a Jordan pair, a triple
system or a Jordan algebra. We prove that the maximal system of
quotients of a strongly nondegenerate Jordan system is the associated system to the maximal
Lie algebra of quotients of its TKK (see Section 4 for the
definitions).

\section{Preliminaries}
In this paper we will deal with Lie algebras over an arbitrary
unital and commutative ring of scalars $\Phi$.

Recall that a $\Phi$-module $L$ together with a bilinear map $[\,
, \,]: L \times L \rightarrow L$, denoted by $(x, y) \mapsto
[x,y]$ (called the {\it bracket of} $x$ and $y$), is said to be a
{\it Lie algebra over} $\Phi$ if the following axioms are
satisfied:
\begin{enumerate}[(i)]\itemsep=2mm
\item[(i)] $[x,x]=0$,
\item[(ii)] $[x,[y,z]] + [y,[z,x]] + [z,[x,y]] =0$ ({\it Jacobi identity}).
\end{enumerate}

Given an abelian group $G$ (whose neutral element will be denoted
by $e$), a Lie algebra $L$ is called $G$-{\it graded} if
$L=\oplus_{\sigma \in \, G}L_\sigma$, where $L_\sigma$ is a
$\Phi$-subspace of $L$ and $[L_\sigma, \, L_\tau]\subseteq
L_{\sigma \tau}$ for all $\sigma, \, \tau \in G$. When the group
is understood we will use the term ``graded'' instead of
``$G$-graded''.  For any subset $X$ of $L$, its support is defined
as $\Supp (X) =\{ \sigma \in G \mid x_\sigma \neq 0 \hbox{ for
some }  x\in X\}$. The grading on $L$ is called {\it finite} if
$\Supp (L)$ is a finite set.

Every Lie algebra $L$ can be seen as a graded algebra over any group $G$ by considering the trivial grading, that is $L=\oplus_{\sigma \in \, G}L_\sigma$,
where $L_e=L$ (being $e$ the neutral element of the group $G$) and $L_\sigma = \{0\}$ for any $\sigma\in G\setminus \{e\}$.
 Thanks to this fact, every statement for graded algebras can be read in terms of algebras.

For a graded Lie algebra $L=\oplus_{\sigma \in \,G}L_\sigma$, the
set of \emph{homogeneous elements} is $\bigcup_{\sigma\in G}
L_\sigma$. The elements of $L_\sigma$ are said to be homogeneous
of \emph{degree} $\sigma$. Graded subalgebras are defined in the natural way. 

An ideal $I$ of a graded Lie algebra $L=\oplus_{\sigma \in
G}L_\sigma$ is called a \emph{graded ideal} if whenever $y=\sum
y_\sigma\in I$ we have $y_\sigma \in I$, for every $\sigma \in G$.
It is
straightforward to show that the sum, the intersection and the
product of graded ideals are again graded ideals.

Let $X$ and $Y$ be two subsets of a Lie algebra $L$. The set
$$\Ann_X(Y):=\{x \in X \mid [x, y]=0 \ \mbox{for
every} \ y \in Y\}$$ is called the \textit{annihilator of} $Y$
\textit{in} $X$, while we will refer to
$$\QAn_X(Y):=\{x \in X \mid  [x, [x, y]]=0 \ \mbox{for
every} \ y \in Y\}$$ as the \textit{quadratic annihilator of} $Y$
\textit{in} $X$. It is easy to check, by using the Jacobi
identity, that $\Ann_L(X)$ is a (graded) ideal of $L$ when $X$ is
also a (graded) ideal of $L$, although the quadratic annihilator
of an ideal needs not be an ideal (see \cite[Examples
1.1]{PS_SND}). In the special situation that $X=L=Y$, $\Ann_L(L)$
is called the {\it center} of $L$ and is denoted by $Z_{L}$; the
elements of the center are called {\it total zero divisors}. If
there is no risk of confusion, we will write $\Ann(Y)$ for
$\Ann_L(Y)$.

We say that a (graded) Lie algebra $L$ is {\it (graded) semiprime}
if for every nonzero (graded) ideal $I$ of $L$, $[I, I] \neq 0$.
In the sequel we shall usually denote $[I,I]$ by $I^2$. Next, $L$
is said to be {\it (graded) prime} if for nonzero (graded) ideals
$I$ and $ J$ of $L$, $[I, J]\neq 0$. A (graded) ideal
$I$ of $L$
is said to be {\it (graded) essential} if its intersection with
any nonzero (graded) ideal is again a nonzero (graded) ideal.

\begin{lem} \label{gr-ann}
Let  $I$ be a graded ideal of a graded Lie algebra
$L=\oplus_{\sigma \in \,G}L_\sigma$. Then:
\begin{enumerate}[(i)]\itemsep=2mm
\item[{\rm (i)}] $\Ann (I)$ is a graded Lie ideal of $L$. In
particular, $Z_L$ is a graded ideal of $L$.
\end{enumerate}
If moreover $L$ is graded semiprime, then:
\begin{enumerate}[(i)]\itemsep=2mm
\item[{\rm20(ii)}
]  $I\cap \Ann (I)=0$.
\item[{\rm (iii)}] $I$ is a graded essential ideal of $L$ if and only if $\Ann (I)=0$.
\end{enumerate}
\end{lem}

\begin{proof} (i). The only thing we are going to show is that every
homogeneous component of any element $x\in \Ann (I)$
is again in $\Ann (I)$. Fix20$\tau\in G$. Note that $[x_\sigma, \,
I_\tau]=0$ for every $\sigma \in G$ because otherwise there would
exist $y_\tau \in I_\tau$ such that $[x_\sigma, \, y_\tau]\neq 0$
for some $\sigma \in G$; this would imply $0\neq [x, \,
I_\tau]\subseteq [x, \, I]=0$, a contradiction. Hence $[x_\sigma,
\, I]= \oplus_{\tau \in \,G}[x_\sigma, \, I_\tau]=0$.

To obtain (ii) and (iii), follow the proofs of conditions (i) and
(ii) in \cite[Lemma 1.2]{msm}.
\end{proof}

Note that if $L$ is (graded) semiprime, then $I^2$ is a (graded)
essential ideal if $I$ is so. Further, the intersection of
(graded) essential ideals is clearly again a (graded) essential
ideal.

Given an element $x$ of a Lie algebra $L$, we may define a map $ \Ad
x\colon L \to L$ by $\Ad x(y)=[x,\, y]$. A (homogeneous) element
$x$ of a (graded) Lie algebra $L$ is a {\it (homogeneous) absolute
zero divisor} if $(\Ad x)^2=0$. The algebra $L$ is said to be {\it
(graded) strongly nondegenerate} if it does not contain nonzero
(homogeneous) absolute zero divisors. It is obvious from the
definitions that (graded) strongly nondegenerate Lie algebras are
(graded) semiprime, but the converse does not hold (see \cite[Remark
1.1]{msm}).

\section{Graded algebras of quotients of graded Lie algebras}

Following the pattern of \cite{msm} it is possible to introduce algebras of quotients of graded Lie algebras, and to build a maximal graded algebra of quotients for every graded semiprime Lie algebra. We devote this section to this end and to relate both notions: the graded and the non-graded ones. Some interesting examples of graded algebras of quotients have also been found.

\begin{definitions}
{\rm Let $L=\oplus_{\sigma \in \,G}L_\sigma$ be a graded
subalgebra of a graded Lie algebra $Q=\oplus_{\sigma \in
\,G}Q_\sigma$.
We say that $Q$ is a {\it graded algebra of quotients} of $L$ or also that $L$ is a {\it graded subalgebra of quotients of} $Q$ if the following 
equivalent conditions are satisfied:
\begin{enumerate}
\item[(i)] Given $0\neq p_\sigma \in Q_\sigma$ and $q_\tau \in Q_\tau$,
there exists $x_\alpha \in L_\alpha$ such that $[x_\alpha, \,
p_\sigma]\neq 0$ and $[x_\alpha, \, _L(q_\tau)]\subseteq L$, where $_L(q_\tau)$ denotes the linear span in $Q$ of $q_\tau$ and the elements of the form $\ad {x_1} \cdots \ad {x_n} q_\tau,$ with $n\in \mathbb{N}$ and $x_1,\dots,  x_n \in L$.
\item[(ii)] $Q$ is graded ideally absorbed into $L$, i.e., for every nonzero element
$q_\tau \in Q_\tau$ there exists a nonzero graded ideal $I$ of $L$
with $\ann L (I)=0$ and such that $0\neq [I, \, q_\tau]\subseteq
L$.
\end{enumerate}
 If for any nonzero $ p_\sigma \in Q_\sigma$ there exists
$x_\alpha \in L_\alpha$ such that $0 \neq [x_\alpha, \,
p_\sigma]\in L$, then we say that $Q$ is a {\it graded weak
algebra of quotients} of $L$, and $L$ is called a {\it graded weak
subalgebra of quotients of} $Q$.}
\end{definitions}

\begin{rems}
{\rm  The definitions before are consistent with the non-graded ones (see \cite[Definitions 2.1 and 2.5]{msm}) in the sense that if $Q$ is a (weak) algebra of quotients of a Lie
algebra $L$, then it is also a graded (weak) algebra of quotients
of $L$ when considering the trivial gradings on $L$ and $Q$. 

The necessary and sufficient condition for a graded Lie algebra to
have a graded (weak) algebra of quotients is the absence of
homogeneous total zero divisors different from zero, condition
that turns out to be equivalent to have zero center.}
\end{rems}

\begin{rem}
{\rm Although every graded algebra of quotients is a graded weak
algebra of quotients, the converse is not true, as shown in the
following example.

Consider the ${\mathbb C}$-module $P$ of all polynomials
$\sum_{r=0}^m \alpha_r x^r$, with $\alpha_i \in {\mathbb C}$ and
$m\in {\mathbb N}$, with the natural ${\mathbb Z}$-grading.
Denote by $\sigma: {\mathbb C}\to {\mathbb C} \ $  the complex
conjugation. Then the following product makes $P$ into a ${\mathbb
Z}$-graded Lie algebra:
$$\left[\sum_{r=0}^m \alpha_r x^r,\ \sum_{s=0}^n \beta_s x^s\right]= \sum_{r, s} (\alpha_r\beta^\sigma_s-\beta_s\alpha_r^\sigma )
x^{r+s}.$$ Let $Q$ be the ${\mathbb Z}$-graded Lie algebra $P/I$,
where $I$ denotes the ${\mathbb Z}$-graded ideal of $P$ consisting
of all polynomials whose first nonzero term has degree at least 4,
and let $L$ be the following graded subalgebra of $Q$:
$$L=\{ \overline\alpha_0 +\overline\alpha_2 \overline x^2 +\overline\alpha_3 \overline x^3\ \vert \
\alpha_0, \alpha_2, \alpha_3 \in {\mathbb C}\},$$ where $\overline
y$ denotes the class of an element $y\in P$ in $P/I$. Then $Q$ is
a graded weak algebra of quotients of $L$, but $Q$ is not a graded
quotient algebra of $L$ since no $l\in L$ satisfies $[l,\overline
x]\in L$ and $[l,\ \overline x^3]\neq 0$ (see \cite[Remark 2.6]{msm}).}
\end{rem}

As it happened in the non-graded case  (see \cite[Proposition 2.7]{msm}), some properties of a graded Lie algebra $L$ are inherited by each of its graded weak algebras of quotients. 
We give here a different approach to show that graded strong nondegeneracy is inherited.

\begin{prop} \label{semiprimeness}
Let $Q=\oplus_{\sigma \in \,G}Q_\sigma$ be a graded weak algebra
of quotients of a graded subalgebra $L$ and suppose that $\Phi$ is 2 and 3-torsion free. Then $L$ graded
strongly nondegenerate implies $Q$ graded strongly nondegenerate.
\end{prop}

\begin{proof}
Suppose that there exists an element $0\neq q_\tau$ in
$Q_\tau$ such that ${(\Ad q_\tau)}^2=0$. Since $Q$ is a graded
weak algebra of quotients of $L$, $0\neq y:=[q_\tau,\,
x_\sigma]\in L$ for some $x_\sigma\in L_\sigma$. As $q_\tau$ is in
$\QAn_Q(Q)\subseteq \QAn_Q(L)$ we have, by \cite[Theorem
2.1]{PS_SND}, that $[y,\, [y, \, u]]\in \QAn_L(L)$ for every $u\in
L$ (observe that the map $u\mapsto \Ad u$  gives an isomorphism
between $L$ and its image inside $A(Q)$, the Lie subalgebra of
$\End (Q)$ generated by the elements $\Ad x$ for $x$ in $Q$; this
allows to apply the result in \cite{PS_SND}). But $\QAn_L(L)$ is
zero, because $L$ is graded strongly nondegenerate, therefore
$[y,\, [y, \, u]]=0$ for every $u\in L$. Again the same reasoning
leads to $y=0$, a contradiction. This shows the statement.
\end{proof}

Now we show the relationship between graded (weak)
algebras of quotients and (weak) algebras of quotients, a useful
tool that combined with other results will provide with examples
of graded algebras of quotients.

\begin{lem} \label{weak quot and weak gr-quot}
Let $L$ be a graded subalgebra of a graded Lie algebra
$Q=\oplus_{\sigma \in \,G}Q_\sigma$. If $Q$ is a weak algebra of
quotients of $L$ then $Q$ is also a graded weak algebra of
quotients of $L$.
\end{lem}

\begin{proof}
For $0\neq q_\tau\in Q_\tau$, apply the hypothesis to find $x\in
L$ such that $0\neq [x,\, q_\tau]\in L$; in particular, $0\neq
[x_\alpha, q_\tau]\in L_{\alpha\tau}$ for some $\alpha \in G$.
\end{proof}

Recall that given a subalgebra $L$ of a Lie algebra $Q$ and an
element $q\in Q$, the set $$(L:\, q):=\{x\in L \mid [x,\,
_L(q)]\subseteq L\}$$ is an ideal of $L$ (see \cite[Lemma 2.10
(i)]{msm}). In case of being $L$ a graded subalgebra of a graded Lie algebra $Q$ and $q_\tau \in Q$, it can be proved $(L:\, q_\tau)$ is indeed a graded ideal of $L$.

\begin{lem} \label{comun property}
Let $Q=\oplus_{\sigma \in \,G}Q_\sigma$ be a graded algebra of
quotients of a graded semiprime Lie algebra $L$. Then, given
$0\neq p_\sigma \in Q_\sigma$ and $q_{\tau_i}\in Q_{\tau_i}$, with
$\tau_i\in G$ and $i=1, \ldots, n$ (for any $n\in \mathbb{N}$),
there exist $\alpha \in G$ and $x_\alpha \in L_\alpha$ such that
$[x_\alpha, \, p_\sigma]\neq 0$ and $[x_\alpha, \,
_L(q_{\tau_i})]\subseteq L$ for every $i=1, \ldots, n$.
\end{lem}

\begin{proof}

Consider  $0\neq p_\sigma \in Q_\sigma$ and $q_{\tau_i}\in
Q_{\tau_i}$,  with $i=1, \ldots, n$. By \cite[Lemma 2.10 (i)]{msm},
$(L:\, q_{\tau_i})$ is a graded essential ideal of $L$ for every
$i$, hence $I=\cap^{n}_{i=1}(L:\, q_{\tau_i})$ is again a graded
essential ideal of $L$. Condition (iii) in Lemma \ref{gr-ann}
implies $\ann L (I)=0$ and by  \cite[Lemma 2.11]{msm} we obtain
$\ann Q (I)=0$. So, there exists $x\in I$ such that
$[x,\,p_\sigma]\neq 0$, and if we decompose $x$ into its
homogeneous components we find some $\alpha \in G$ satisfying
$[x_\alpha, \, p_\sigma]\neq 0$. Now the proof is complete because
$x_\alpha \in I$ as $I$ is a graded ideal and $x\in I$.
\end{proof}

\begin{prop} \label{quot and gr-quot}

Let $L$ be a graded subalgebra of a graded Lie algebra
$Q=\oplus_{\sigma \in \,G}Q_\sigma$. Consider the following
conditions:
\begin{enumerate}

\item[\rm{(i)}] $Q$ is an algebra of quotients of $L$.
\item[\rm{(ii)}] $Q$ is a graded algebra of quotients of $L$.
\end{enumerate}

Then (i) implies (ii). Moreover, if $L$ is graded semiprime then
 (ii) implies (i).
\end{prop}

\begin{proof}
(i) $\Rightarrow$ (ii). Given $0\neq p_\sigma \in Q_\sigma$ and
$q_\tau\in Q_\tau$, by the hypothesis there exists $x\in L$
satisfying $[x,\, p_\sigma]\neq 0$ and $[x,\, _L(q_\tau)]\subseteq
L$, that is, $x\in (L:\, q_\tau)$. This means by the considerations above
that $x_\alpha \in (L:\, q_\tau)$.

(ii) $\Rightarrow$ (i). Suppose now that $Q$ is a graded algebra
of quotients of $L$, with $L$ graded semiprime. Take $p,\, q$ in
$Q$, with $p\neq 0$; let $\sigma\in G$ be such that
$p_{\sigma}\neq 0$ and write $\tau_1,\,\tau_2,\ldots,\tau_n$ to
denote the elements of $\Supp(q)$.
 By Lemma
\ref{comun property} it is possible to find an element
$x_\alpha\in L_\alpha$ satisfying $[x_\alpha,\, p_{\sigma}]\neq 0$
and $[x_\alpha, \, _L(q_{\tau_i})]\subseteq L$ for every
$i=1,\ldots, n$, hence $[x_\alpha, p]\neq 0$ and $[x_\alpha, \,
_L(q)]\subseteq L$; this shows that $Q$ is an algebra of quotients
of $L$.
\end{proof}

We continue the section with some important examples of graded
algebras of quotients of graded Lie algebras. For brevity, we will not include all the definitions involved in
the first example (we refer the reader to \cite[5.4]{damour}). However, in Section 4 we will explain what the $\tkk$-algebra of a
Jordan pair is.
 Recall
that any strongly prime hermitian Jordan pair $V$ is sandwiched as
follows (see \cite[5.4]{damour}):
$$H(R, \,\ast) \lhd V \leq H(Q(R), \ast),$$
where $R$ is a $\ast$-prime associative pair with involution and
$Q(R)$ is its associative Martindale pair of symmetric quotients.

\begin{ejm}
{\rm   Let $R$ be a $\ast$-prime associative pair with involution,
and $Q(R)$ its Martindale pair of symmetric quotients. Then
$TKK(H(Q(R), \ast))$ is a 3-graded algebra of quotients of
$TKK(H(R, \ast))$.}
\end{ejm}
\begin{proof}
From \cite[Proposition 4.2 and Corollary 4.3]{esther},
$Q:=TKK(H(Q(R), \ast))$ is ideally absorbed into the strongly
prime Lie algebra $L:=TKK(H(R, \ast))$; use \cite[Proposition
2.15]{msm} to obtain that $Q$ is an algebra of quotients of $L$
and Proposition \ref{quot and gr-quot} to reach the conclusion.
\end{proof}

Now, we provide with examples of graded Lie algebras of quotients
of graded Lie algebras arising from associative algebras graded by
abelian groups.

Let $A$ be an associative algebra with an involution $*$; then the
set of skew elements $K_A= \{ x \in A ~ : ~ x^* = -x \}$ is a Lie
subalgebra of $A^-$. If $A$ is $G$-graded and
$A_\sigma^\ast=A_\sigma$, for all $\sigma \in G$, then $K_A$ and
$K_A/Z_{K_A}$ are $G$-graded Lie algebras too.

For a semiprime associative algebra, denote by $Q_s(A)$ its
Martindale symmetric ring of quotients (see, for example
\cite{BMM} for more information about this ring of quotients).

\begin{thm}\label{algcocientegK} Let $A$ be a  semiprime $G$-graded
associative algebra with an involution $*$ such that
$A_\sigma^\ast=A_\sigma$, for every $\sigma \in G$, and let
$Q=\oplus_{\sigma \in \,G}Q_\sigma$
 be a $G$-graded overalgebra of $A$ contained in $Q_s(A)$
 and satisfying $Q_\sigma^\ast=Q_\sigma$ for every $\sigma \in G$.
Then:
\begin{enumerate}[(i)]\itemsep=2mm
\item[\rm{(i)}] $K_Q/Z_{K_Q}$ is a graded algebra of quotients of
$K_A/Z_{K_A}$.
\item[\rm{(ii)}] $[K_Q,K_Q]/ Z_{[K_Q,K_Q]}$ is a graded  algebra of quotients of $[K_A, K_A]/Z_{[K_A,K_A]}$.
\end{enumerate}
\end{thm}
\begin{proof}
By \cite[Theorem 1]{CabreraSanchez}, $K_Q/Z_{K_Q}$ and $[K_Q,K_Q]/
Z_{[K_Q,K_Q]}$ are algebras of quotients of $K_A/Z_{K_A}$ and
$[K_A, K_A]/Z_{[K_A,K_A]}$, respectively.  Since $A$ is semiprime,
and so is $Q$ (by \cite[Lemma 2.1.9.(i)]{BMM}), it follows from
\cite[Theorem 6.1]{marmi} that $K_A/Z_{K_A}$ and $K_Q/Z_{K_Q}$ are
semiprime Lie algebras. In particular, they are graded semiprime,
hence Proposition \ref{quot and gr-quot} applies to get the
result.
\end{proof}

Note that for an arbitrary graded associative algebra $A$, the
graded Lie algebra $A^-$ is graded isomorphic to $K_{A \oplus
A^0}$ and hence $A^- / Z_A$ is  graded isomorphic to $K_{A \oplus
A^0}/Z_{K_{A \oplus A^0}}$, where $A^0$ denotes the opposite
algebra of $A$, and  $A \oplus A^0$ is endowed with the exchange
involution. This fact allows to obtain the following consequence
of the theorem before.

\begin{cor} Let $A$ be a semiprime graded
associative algebra and $Q$ be a graded subalgebra of $Q_s(A)$
containing $A$. Then
\begin{enumerate}[(i)]\itemsep=2mm
\item[\rm (i)] $Q^-/Z_{Q}$ is a graded  algebra of
quotients of $A^-/Z_{A}$.
 \item[\rm (ii)] $[Q,Q]/ Z_{[Q,Q]}$ is a graded algebra of quotients of $[A,A]/Z_{[A,A]}$.
\end{enumerate}
\end{cor}

\smallskip

Following the construction given in \cite{msm} of the maximal
algebra of quotients of a semiprime Lie algebra, it is possible to build a
maximal graded algebra of quotients for every graded semiprime Lie
algebra. Taking into account that the elements of the maximal algebra of quotients of a semiprime Lie algebra arise from derivations defined on essential ideals, it seems natural to consider instead graded derivations defined on graded essential ideals. With this idea in mind, we proceed to introduce a new graded algebra of quotients.

Let $I$ be an ideal of a  Lie algebra $L$. A linear map
$\delta:I\rightarrow L$ is said to be a {\it derivation of
L} if $$\delta([x,\, y])=[\delta x,\, y]+[x,\,\delta y]$$ for
every $x,\, y\in I$, and $\der(I, L)$ will stand for the set of  all derivations
from $I$ into $L$. Suppose now that the Lie algebra $L$ is graded by a group $G$, and
that $I$ is a graded ideal of $L$. We say that a derivation $\delta$ has {\it degree} $\sigma \in G$ if it satisfies
$\delta (I_\tau)\subseteq L_{\tau \sigma}$ for every $\tau\in G$.
In this case, $\delta$ is called a {\it graded derivation
of degree} $\sigma$. Denote by $\dergr (I,\, L)_\sigma$ the set
of all graded derivations of degree $\sigma$. Clearly, it becomes a $\Phi$-module by defining operations in
the natural way and, consequently, $\dergr
(I,\,L):=\oplus_{\sigma\in G}\dergr (I,\, L)_\sigma $ is also a
$\Phi$-module.

For any element $x$ in a Lie algebra $L$, the
adjoint map $\Ad x:I\rightarrow L$ defined by $\Ad x (y)=[x,\,y]$
is a derivation of $L$. If $L$ is graded by a group $G$
and $x$ is a homogeneous element of degree $\sigma$,
 then $\Ad x$ is in fact a derivation of degree
$\sigma$. In general, for any $x$ in the graded Lie algebra $L$,
 $$\Ad x=\sum_{\sigma\in G}
\Ad x_\sigma \in\bigoplus_{\sigma\in G}  \dergr (I,\, L)_\sigma
= \dergr (I,\,L).$$

\smallskip

In order to ease the notation, denote by $\Igre (L)$ the set of
all graded essential ideals of a graded Lie algebra $L$. If $L$ is a $G$-graded semiprime Lie algebra, it can be shown (as in \cite[Theorem 3.4]{msm}) that the direct limit 
$$Q_{gr-m}(L):={\underset {I\in \Igre (L)}{\underrightarrow{\limit}}}{\dergr (I,\, L)}$$
 of graded derivations of $L$ defined on graded essential ideals of $L$ is a graded algebra of quotients of $L$ containing $L$ as a graded subalgebra, via the following graded Lie monomorphism: $$\begin{matrix}
\varphi: & L & \to & Q_{gr-m}(L)&\cr & x & \mapsto & (\Ad x)_L
\end{matrix}$$
\noindent where $\delta_I$ stands for any arbitrary element of $Q_{gr-m}(L)$.

Moreover, $Q_{gr-m}(L)$ is maximal among the
graded algebras of quotients of $L$. It is called the {\it maximal graded algebra of quotients} of $L$. The following result characterizes it. We omit its proof because it is similar to that of \cite[Theorem 3.8]{msm}.

\begin{thm} \label{axiomatic charact}
Let $L$ be a graded semiprime Lie algebra and consider a graded
overalgebra $S$ of $L$. Then $S$ is graded isomorphic to
$Q_{gr-m}(L)$, under an isomorphism which is the identity on $L$,
if and only if $S$ satisfies the following properties:
\begin{enumerate}

\item[\rm{(i)}] For any $s_\sigma \in S_\sigma$ ($\sigma \in G$) there exists
$I\in \Igre (L)$ such that $[I,\, s_\sigma]\subseteq L$.
\item[\rm{(ii)}] For $s_\sigma \in S_\sigma$  ($\sigma \in G$) and $I\in \Igre (L)$,
$[I,\, s_\sigma]=0$ implies $s_\sigma=0$.
\item[\rm{(iii)}] For $I\in \Igre (L)$ and $\delta\in \pdergr (I,\,
L)_\sigma$ ($\sigma \in G$) there exists $s_\sigma\in S_\sigma$
such that $\delta (x)=[s_\sigma,\, x]$ for every $x\in I$.
\end{enumerate}

\end{thm}

\begin{rem}

{\rm Note the conditions (i) and (ii) in the theorem above are
equivalent to the following one:

(ii)$^\prime$ $S$ is a graded algebra of quotients of $L$.}

\end{rem}

\begin{rem}

{\rm The notion of maximal graded algebra of quotients extends
that of maximal algebra of quotients given in \cite{msm} as the
maximal graded algebra of quotients and the maximal algebra of
quotients of a semiprime Lie algebra coincide when considering the
trivial grading over such an algebra.}

\end{rem}

\section{Maximal graded algebras of quotients of
3-graded Lie algebras}

Let $L$ be a $\Z$-graded Lie algebra with a finite grading. We may
write $L=\oplus^{n}_{k=-n}L_k$ and we will say that $L$ has a
$(2n+1)$-{\it grading}. In what follows, we will deal with 3-graded
Lie algebras.

In this section we show that for a 3-graded semiprime Lie algebra
$L$, the maximal algebra of quotients of $L$ is 3-graded too and
coincides with the maximal graded algebra of quotients of $L$, as
defined in Section 4.

\begin{lem}\label{gradedideal}

Let $L=L_{-1}\oplus L_0\oplus L_1$ be a 3-graded Lie algebra and
$I$ an ideal of $L$. Denote by $\pi_i$ the canonical projection
from  $L$ into $L_i$ (with $i\in \{-1,\, 0,\,1\}$) and consider
$\tilde{I}:=J+\pi_{-1}(J)+\pi_1(J)$, where $J:=[[I, \,
I],\,[I,\,I]]$. Then:
\begin {enumerate}[(i)]\itemsep=2mm
\item[\rm{(i)}] $\tilde{I}$ is a graded ideal of $L$ contained in $I$.
\end{enumerate}
\noindent If moreover $L$ is semiprime, then:
\begin {enumerate}[(i)]\itemsep=2mm
\item[\rm{(ii)}] $I$ is an essential ideal of $L$ if and only if $\tilde{I}$ is an essential ideal of $L$.
\item[\rm{(iii)}] Suppose that $I$ is a graded ideal. Then $I$ is an essential ideal of $L$ if and only if it
 is a graded essential ideal of $L$.
\end{enumerate}
\end{lem}
\begin{proof}

(i). Note that $\pi_0(J) \subseteq \tilde{I}$ since $\pi_0={\rm
Id}-\pi_{-1}-\pi_1$. Show first that $\tilde{I}$ is an ideal of
$L$: take $x\in \tilde{I}$ and $y\in L$ and write
$x=u+z_{-1}+t_1$, where $u$ and the elements $z=z_{-1}+z_0+z_1$
and $t=t_{-1}+t_0+t_1$ are in $J$. We have
\[ \label{grad}
\tag{\ensuremath{\dagger}} [x,\,y]=[u,\,y]+[z_{-1},\,y]+[t_1,\,y].
\]
Now, since $u$ is in $J$, which is an ideal of $L$, we obtain
$[u,\,y]\in J\subseteq \tilde{I}$. On the other hand, writing
$y=y_{-1}+y_0+y_1$ we have
$[z_{-1},\,y]=[z_{-1},\,y_0]+[z_{-1},\,y_1]$; apply again that $J$
is an ideal to obtain $[z,\, y_1],\, [z,\,y_0] \in J$, which
implies that the elements $[z,\, y_1]_0=[z_{-1},\,y_1]$ and
$[z,\,y_0]_{-1}=[z_{-1},\,y_0]$ are in $\tilde{I}$. Hence,
$[z_{-1},\,y]\in \tilde{I}$. Analogously, it can be shown
$[t_1,\,y]\in \tilde{I}$. Put together (\ref{grad}) and this to
obtain $[x,\,y]\in \tilde{I}$, as desired.

We claim that $\tilde{I}$ is in fact a graded ideal: consider $xx_{-1}+x_0+x_1\in \tilde{I}$ and write, as above,
$x=u+z_{-1}+t_1$, with $u$, $z=z_{-1}+z_0+z_1$ and
$t=t_{-1}+t_0+t_1$ elements in $J$. Then $x_{-1}=u_{-1}+z_{-1}$,
$x_0=u_0$ and $x_1=u_1+t_1$. Thus, taking into account the
definition of $\tilde{I}$ we obtain that $x_i\in \tilde{I}$ for
$i\in \{-1,0,1\}$.

 Finally, we prove that $\tilde{I}$ is contained in $I$ by showing that
$\pi_{-1}(J)$ and $\pi_1(J)$ are contained in $I$. Define
$\delta:=\pi_1-\pi_{-1}$. Then $\delta^2=\pi_{-1}+\pi_1$ implies
$2\pi_1=\delta^2+\delta$ and $2\pi_{-1}=\delta^2-\delta$. Hence,
to prove that $\pi_{-1}(J)$ and $\pi_1(J)$ are contained in $I$,
it is enough to check that $\delta^2(J)$ and $\delta(J)$ are
contained in $I$. Take $x,\, y\in I$ and write $x=x_{-1}+x_0+x_1$
and $y=y_{-1}+y_0+y_1$ where $x_i,\,y_i\in L_i$ for $i\in
\{-1,0,1\}$. A computation gives
\[
\begin{split}
&[x,\,y]_{-1}=[x_{-1},\,y_0]+[x_0,\,y_{-1}]=[x_{-1},\,y]-
[x_{-1},\,y_1]+[x,\,y_{-1}]-[x_1,\,y_{-1}]\\&
[x,\,y]_1=[x_0,\,y_1]+[x_1,\,y_0]=[x_1,\,y]-[x_1,\,y_{-1}]+[x,\,y_1]-[x_{-1},\,y_1].
\end{split}
\]
So
$\delta([x,\,y])=[x,\,y]_1-[x,\,y]_{-1}=[x_1,\,y]+[x,\,y_1]-[x_{-1},\,y]-[x,\,y_{-1}]\in
I$, that is, $\delta([I,\,I])\subseteq I$; it can be proved
analogously $\delta (J)\subseteq [I,\,I]\subseteq I$, therefore
$\delta^2(J)\subseteq \delta([I,\,I])\subseteq I$.

\medskip

(ii). Consider $I$ as an essential ideal of $L$. Note that the
semiprimeness of $L$ implies that $J$ is also an essential ideal
of $L$. Hence, $J\cap K\neq 0$ for any nonzero ideal $K$ of $L$
and so $\tilde{I}\cap K\neq 0$. This shows that $\tilde{I}$ is an
essential ideal of $L$.

To prove the converse, suppose that $\tilde{I}$ is an essential
ideal of $L$. As $\tilde{I} \subseteq I$ (by (i)), the ideal $I$
must be essential too.

\medskip
(iii). It is trivial that $I$ essential as an ideal implies $I$
essential as a graded ideal.

Suppose now that $I$ is a graded essential ideal and let $U$ be a
nonzero ideal of $L$. Being $L$ semiprime,
$K:=[[U,\,U],\,[U,\,U]]$ is a nonzero ideal of $L$. Apply (i) to
obtain that $\tilde{U}:=K+\pi_{-1}(K)+\pi_1(K)$ is a graded ideal
of $L$ contained in $U$. As $I$ is a graded essential ideal,
$I\cap \tilde{U}\neq 0$ and hence $I\cap U\neq 0$.
\end{proof}

\begin{thm} \label{isom between Q_m and Q_{gr-m}}
Let $L=L_{-1}\oplus L_0\oplus L_1$ be a semiprime 3-graded Lie
algebra. Then:
\begin{enumerate}[(i)]\itemsep=2mm
\item[\rm{(i)}]   $Q_m(L)$ is  graded and graded isomorphic to $Q_{gr-m}(L).$
\item[\rm{(ii)}] If $L$ is strongly nondegenerate and $\Phi$ is 2 and 3-torsion free, then $Q_m(L)$ is a
3-graded strongly nondegenerate Lie algebra.
\end{enumerate}
\end{thm}

\begin{proof}
(i). Observe first that $L$, viewed as a 3-graded Lie algebra, is graded
semiprime (since it is semiprime), so it has sense to consider
$Q_{gr-m}(L)$.  

Define
$$\begin{matrix} \varphi: & Q_m(L) & \to & Q_{gr-m}(L)\cr & \delta_I & \mapsto & \delta_{\tilde{I}}

\end{matrix}$$
where for an essential ideal  $I$ of $L$, $\tilde{I}\subseteq I$
is the graded essential ideal defined in Lemma \ref{gradedideal}.

The Lie algebra $Q_m(L) $ is 5-graded: let $I$ be an essential ideal of $L$ and
$\tilde{I}$ as before. It is easy to check, by
considering the canonical projections onto the subspaces $\tilde{I}_i$
($i\in \{-1,0,1\}$), that $\pder (\tilde{I},\,L)$ is just
$\oplus_{i=-2}^2\pdergr (\tilde{I},\,L)_i$, which coincides, by
definition, with $\pdergr (\tilde{I},\, L)$. This shows our claim and that
the map $\varphi$ is well-defined.
Finally, keeping in mind Lemma \ref{gradedideal} (iii) it is straightforward to verify that $\varphi$ is a graded Lie
algebra isomorphism. 

(ii). Apply (i) and \cite[Proposition 1.7]{gg}.
\end{proof}

\section{Jordan pairs of
quotients and 3-graded Lie algebras of quotients}

Our target in this section is to analyze the relationship between
the notions of Jordan pair of quotients 
and of (graded) Lie algebra of quotients, via the
Tits-Kantor-Koecher construction. In the particular case of
maximal quotients,we will prove that
under the suitable hypothesis, the maximal Lie algebra
of quotients of  the TKK-algebra of a Jordan pair  $V$ is the
TKK-algebra of the maximal Jordan pair of quotients of $V$.

A {\it Jordan pair over} $\Phi$ is a pair $V=(V^+,\,V^-)$ of
$\Phi$-modules together with a pair $(Q^+,\,Q^-)$ of quadratics
maps $Q^\sigma:V^\sigma\rightarrow \Hom (V^{-\sigma},\,V^\sigma)$
(for $ \sigma=\pm$) with linearizations denoted by
$Q^\sigma_{x,z}\,y =\{x,\,y,\,z\}=D^\sigma_{x,y}z$, where
$Q^\sigma_{x,z}=Q^\sigma_{x+z}-Q^\sigma_x-Q^\sigma_z$, satisfying
the following identities in all the scalar extensions of  $\Phi$:

\begin{enumerate}\itemsep=2mm
\item[(i)] $D^\sigma_{x,y}\,Q^\sigma_x=Q^\sigma_x\,D^{-\sigma}_{y,x}$
\item[(ii)] $D^\sigma_{Q^\sigma_x y,\,y}=D^\sigma_{x,\,Q^{-\sigma}_y x}$
\item[(iii)] $Q^\sigma_{\,Q^\sigma_x y}=Q^\sigma_x Q^{-\sigma}_y Q^\sigma_x$
\end{enumerate}

for every $x\in V^\sigma$ and $y\in V^{-\sigma}$.

From now on, we shall deal with Jordan pairs $V=(V^+,\,V^-)$ over
a ring of scalars $\Phi$ containing $\frac{1}{2}$.  In order to
ease the notation, Jordan products will be denoted by $Q_x y$, for
any $x\in V^\sigma$, $y\in V^{-\sigma}$. The reader is referred to
\cite{loos} for basic results, notation and terminology on Jordan
pairs. Nevertheless, we recall here some notions and basic
properties.

Let $V=(V^+,\,V^-)$ be a Jordan pair. An element $x\in V^\sigma$
is called an  \emph{absolute zero divisor} if $Q_x=0$, while the
pair $V$ is said to be \emph{strongly nondegenerate}
(\emph{nondegenerate} in the terminology of \cite{gg}) if it has
no nonzero absolute zero divisors. The pair $V$ is
{\emph{semiprime}} if $Q_{I^\pm}I^{\mp}=0$ imply $I=0$, being $I$
an ideal of $V$, and is called {\it prime} if $Q_{I^\pm}J^{\mp}=0$
imply $I=0$ or $J=0$, for $I$ and $J$ ideals of $V$. A {\it
strongly prime} pair is a prime and strongly nondegenerate pair.

For a subset $X=(X^+,\,X^-)$ of $V$, the annihilator of $X$ in $V$
is $\Ann _V (X)=(\Ann _V (X)^+,$ $\,\Ann _V (X)^-)$, where, for
$\sigma= \pm$
$$\Ann _V (X)^\sigma=\{z\in V^\sigma \mid \{z,\,X^{-\sigma},\,
V^\sigma\}= \{z,\,V^{-\sigma},\, X^\sigma\}=\{V^{-\sigma},\,z,\,
X^{-\sigma}\}=0\}.$$ One can check that $\Ann _V(I)$ is an ideal
of $V$ if $I$ is so. Ideals of Lie algebras having zero
annihilator are essentials and when the Lie algebra where they
live is semiprime, the reverse holds, i.e., every essential ideal
has zero annihilator (see \cite[Lemma 1.2]{msm}). In the context
of Jordan pairs, a similar result can be shown.

\begin{lem} \label{essential=sturdy}
Let $I=(I^+,\,I^-)$ be an ideal of a semiprime Jordan pair
$V=(V^+,\,V^-)$. Then:
\begin{enumerate}
\item[\rm{(i)}] $I\cap \Ann _V(I)=0$.
\item[\rm{(ii)}] $I$ is an essential ideal of $V$ if and only if
$\Ann _V(I)=0$.
\end{enumerate}
\end{lem}

\begin{proof}
(i). If we show that the ideal $K:=I\cap \Ann _V(I)$ satisfies
that $Q_{K^\pm}K^{\mp}=0$, for $K^\sigma=I^\sigma\cap
\Ann_V(I)^\sigma$, $\sigma= \pm$, the result follows by the
semiprimeness of $V$. Given $x\in K^{\sigma}\subseteq \Ann
_V(I)^{\sigma}$ for $\sigma=\pm$ we have
$\{x,\,K^{-\sigma},\,V^{\sigma}\}=0$ since $K^{-\sigma}\subseteq
I^{-\sigma}$. So $\{K^{\sigma},\,K^{-\sigma},\,V^{\sigma}\}=0$ for
$\sigma=\pm$ and hence $Q_{K^\pm}K^{\mp}=0$, as desired.

(ii). Consider an essential ideal $I=(I^+,\,I^-)$ of $V$; then
$I\cap\Ann_V(I)=0$ by (i), and by the essentiality, $\Ann
_V(I)=0$. Conversely, suppose that $\Ann_V(I)=0$ and consider  an
ideal $K=(K^+,\,K^-)$ of $V$ satisfying $I\cap K=0$.

For $x\in K^\sigma$, with $\sigma=\pm$, and taking into account
that $I$ and $K$ are ideals of $V$, we obtain
$$\{x,\,I^{-\sigma},\,V^\sigma\},
\{x,\,V^{-\sigma},\,I^\sigma\}, \{V^{-\sigma},\,x,\,I^{-\sigma}\}
\subseteq I\cap K=0,$$

hence, $K\subseteq \Ann _V(I)=0$. This shows that $I$ is an
essential ideal of $V$.
\end{proof}

Let us recall the connection between Jordan 3-graded Lie algebras
and Jordan pairs.

A 3-graded Lie algebra $L=L_{-1}\oplus L_0\oplus L_1$ is called
\emph{Jordan 3-graded} if $[L_1,\,L_{-1}]=L_0$ and there exists a
Jordan pair structure on $(L_1,\,L_{-1})$ whose Jordan product is
related to the Lie product by $\{x,\,y,\,z\}=[[x,\,y],\,z]$, for
any $x,\,z\in L_\sigma, \, y\in L_{-\sigma}, \, \sigma=\pm$. In
this case, $V=(L_1,\,L_{-1})$ is called the \emph{associated
Jordan pair}.

Since $\frac{1}{2} \in \Phi$, the product on the associated Jordan
pair is unique and given by $Q_x
y=\frac{1}{2}\{x,\,y,\,x\}=\frac{1}{2}[[x,\,y],\,x]$. Conversely,
for any 3-graded Lie algebra, the formula above defines a pair
structure on $(L_1,\,L_{-1})$ whenever $\frac{1}{6}\in \Phi$ (see
\cite[1.2]{neher}).

One important example of a Jordan 3-graded Lie algebra is the
$\tkk$-algebra of a Jordan pair. It is built as explained below.

Let $V=(V^+,\,V^-)$ be a Jordan pair; a pair $(\delta^+,\,
\delta^{-}) \in \End_\Phi (V^+) \times \End_\Phi (V^-)$ is a {\it
derivation} of $V$ if it satisfies
$$\delta^\sigma(\{x,\,y,\,z\})=\{\delta^\sigma(x),\,y,\,z\}+
\{x,\,\delta^{-\sigma}(y),\,z\}+\{x,\,y,\,\delta^\sigma(z)\}$$
for any $x,\,z \in V^\sigma$ and $y\in V^{-\sigma}$, $\sigma=\pm$.
For $(x,\,y)\in V$ the map $\delta(x,\,y):=(D_{x,y},\,-D_{y,x})$
is a derivation of $V$ (by the identity (JP12) in \cite{loos})
called {\it inner derivation}. Denote by $\ider (V)$ the
$\Phi$-module spanned by all inner derivations of $V$ and define
on the $\Phi$-module $\tkk (V):=V^+\oplus \ider (V) \oplus V^-$
the following product:

$$
\begin{aligned}
 { [x^+\oplus \,\gamma \,\oplus
\,x^-,\,y^+\oplus\, \mu\,\oplus \, y^-] =} &
(\gamma_+x^+-\mu_+x^+)\,\oplus\,
(\,[\gamma,\,\mu]\,+\,\delta(x^+,\,y^-)\cr
 -&\,\delta(y^+,\,x^-)\,)\, \oplus\, (\gamma_-y^-\mu_-x^-),
\end{aligned}
$$ where $x^\sigma$, $y^\sigma \in V^\sigma$ and
$\gamma=(\gamma_+,\,\gamma_-),\,\mu=(\mu_+,\,\mu_-)\in \ider (V)$.
Then, it can be proved that $\tkk (V)$ becomes a Lie algebra (see
e.g. \cite{Mey}). As this construction has its origin in the
fundamental papers \cite{kantor64, kantor67, kantor72} by Kantor,
in \cite{koecher67, koecher68} by Koecher and in \cite{tits} by
Tits, $\tkk (V)$ is called the \emph{Tits-Kantor-Koecher algebra
of} $V$ or the $\tkk$-\emph{algebra} for short. It is easy to
check that the following provides $\tkk (V)$ with a 3-grading:
$$\tkk(V)_1=V^+,\,\,\tkk(V)_0=\ider (V),\,\,\tkk(V)_{-1}=V^-.$$
Moreover, $\tkk (V)$ is a Jordan 3-graded Lie algebra with $V$ as
associated Jordan pair.

If $L$ is a Jordan 3-graded Lie algebra with associated Jordan
pair $V$, then the $\tkk$-algebra associated to $V$ is not in
general isomorphic to $L$. Rather, we have:

\begin{lem} \label{neher93}
{\rm (\cite[2.8]{neher93})}. Let $L$ be a Jordan 3-graded Lie
algebra with associated Jordan pair $V$. Then $\tkk (V)\cong
L/C_V$, where $C_V=\{x\in L_0\mid
[x,\,L_1]=0=[x,\,L_{-1}]\}=Z(L)\cap L_0$.
\end{lem}

Now, we are going to show the equivalence between Jordan pairs of
quotients and Lie algebras of quotients of their respective
$\tkk$-algebras.

\begin{lem}\label{tkk filters}
Let $V$ be a semiprime Jordan pair, and $I=(I^+, I^-)$ an ideal of
$V$. Define by ${\rm Id}_{\tkk(V)}(I)= I^+\oplus
(\,[I^+,\,V^-]+[V^+,\,I^-]\,)\oplus I^-$ the graded ideal of $\tkk
(V)$ generated by $I$. Then $\ann{\tkk(V)}({\rm
Id}_{\tkk(V)}(I))=0$ if and only if $\ann V(I)=0$.
\end{lem}
\begin{proof}
See \cite[Lemma 2.9]{gg}.
\end{proof}

\smallskip

\begin{definition}{\rm (See \cite[2.5]{gg}). Let $V$ be a semiprime
Jordan pair contained in a Jordan pair $W$. It is said that $W$ is a
{\emph{pair of $\,\mathfrak{M}$-quotients}} of $V$ if for every
$0\neq q\in W^\sigma$ (with $ \sigma=\pm$) there exists an ideal $I$
of $V$ with $\ann V(I)=0$ such that $\{q,\,I^{-\sigma},\,V^\sigma
\}+\{q,\,V^{-\sigma},\,I^{\sigma}\}\subseteq V^\sigma$ and
$\{I^{-\sigma},\,q,\,V^{-\sigma}\}\subseteq V^{-\sigma}$, with
either $\{q,\,I^{-\sigma},\,V^\sigma
\}+\{q,\,V^{-\sigma},\,I^{\sigma}\}\neq 0$ or
$\{I^{-\sigma},\,q,\,V^{-\sigma}\}\neq 0$.}
\end{definition}

\begin{thm} \label{equiv Jordan pair and Lie quotients}
Let $V$ be a semiprime subpair of a Jordan pair $W$. Then the
following conditions are equivalent:
\begin{enumerate}
\item[\rm{(i)}] $W$ is a pair of $\,\mathfrak{M}$-quotients of $V$. \smallskip
\item[\rm{(ii)}] $\tkk(W)$ is an algebra of quotients of $\tkk
(V)$.
\end{enumerate}
\end{thm}

\begin{proof}
(i) $\Rightarrow$ (ii) is \cite[Theorem 2.10]{gg}.

(ii) $\Rightarrow$ (i). Take $0\neq q^\sigma \in W^{\sigma}$
($\sigma=\pm$) and apply Propositions \ref{quot and gr-quot}  to find a 3-graded ideal $I$ of $\tkk (V)$ with
$\Ann _{\tkk (V)}(I)=0$ and such that $0\neq
[I,\,q^{\sigma}]\subseteq \tkk (V)$. We claim that $I_V:=I_1\oplus
(\,[I_1,\,V^-]+[V^+,\,I_{-1}]\,)\oplus I_{-1}$ is an essential
ideal of $\tkk(V)$, where $I=I_1\oplus I_0\oplus I_{-1}$. Let
$K=K_1\oplus K_0\oplus K_{-1}$ be a nonzero 3-graded ideal of
$\tkk (V)$; the semiprimeness of $V$ implies that either $I_1\cap
K_1\neq 0$ or $I_{-1}\cap K_{-1}\neq 0$ (see the proof of
\cite[Proposition 2.6]{gn}) and therefore $I_V\cap K\neq 0$. By
Lemma and \ref{gr-ann} (iii), $\ann {\tkk (V)} (I_V)=0$, and by
Lemma \ref{tkk filters}, $\ann{V}((I_1, I_{-1}))=0$.

Denote $I_1$ and $I_{-1}$ by $I^+$ and $I^-$, respectively. Then,
for $\sigma=\pm$ we have:
$$
\begin{aligned}
{\{q^\sigma,\,I^{-\sigma},\,V^\sigma\}} & \subseteq
[\,[q^\sigma,\,I^{-\sigma}],\,V^\sigma]\subseteq V^\sigma \cr
\{I^{-\sigma},\,q^\sigma,\,V^{-\sigma}\} &\subseteq
[\,[I^{-\sigma}, \,q^\sigma],\,V^{-\sigma}]\subseteq V^{-\sigma}
\cr \{q^\sigma,\,V^{-\sigma},\,I^{\sigma}\}&\subseteq
 [\,[q^\sigma,\,V^{-\sigma}],\,I^\sigma]\subseteq
[\,[V^{-\sigma},\,I^\sigma],\,q^\sigma]\subseteq V^\sigma
\end{aligned}
$$

To complete the proof we have to check that either
$\{q^\sigma,\,I^{-\sigma},\,V^\sigma
\}+\{q^\sigma,\,V^{-\sigma},\,I^{\sigma}\}\neq 0$ or
$\{I^{-\sigma},\,q^\sigma,\,V^{-\sigma}\}\neq 0$. We have just
showed that $\Ann_{\tkk (V)}(I_V)=0$; using \cite[Lemma 2.11]{msm}
we obtain that $\Ann_{\tkk (W)}(I_V)=0$ and hence $0\neq
[I_V,\,q^\sigma]\subseteq [I,\,q^\sigma]\subseteq \tkk (V)$ which
implies that either $[(I_V)_0,\,q^\sigma]\neq 0$ or
$[I^{-\sigma},\,q^\sigma]\neq 0$. In the first case, we have:
$$0\neq
[(I_V)_0,\,q^\sigma]=[\,[I^\sigma,\,V^{-\sigma}],\,q^\sigma]+
[\,[V^\sigma,\,I^{-\sigma}],\,q^\sigma]\subseteq
\{I^\sigma,\,V^{-\sigma},\,q^\sigma\}+\{V^\sigma,\,I^{-\sigma},\,q^\sigma\}.$$
In the second case, apply that the representation of $\ider (V)$
on $V$ is faithful to obtain $$
\begin{aligned}
{0\neq [\,[I^{-\sigma},\,q^\sigma], \, V^{-\sigma}]} = \,& \{I^{-\sigma},\,q^\sigma, \, V^{-\sigma}\} \, \, \, \mbox{ or}\cr
0\neq [\,[I^{-\sigma},\,q^\sigma],\,V^\sigma]\subseteq
[\,[V^\sigma,\,I^{-\sigma}],\,q^\sigma]&
=\{V^\sigma,\,I^{-\sigma},\,q^\sigma\}= \{q^\sigma,\,I^{-\sigma},\,V^\sigma\}.
\end{aligned} $$
\end{proof}

We continue the section by examining the relationship between
maximal Jordan pairs of $\mathfrak{M}$-quotients (see \cite[3.1 and
Theorem 3.2]{gg} for definition) and maximal algebras of quotients
of Jordan 3-graded Lie algebras.

\begin{lem} \label{sturdy filter of TKK}
Let $V=(V^+,\,V^-)$ be a strongly nondegenerate Jordan pair. If
$I$ is an essential ideal of $\tkk(V)$, then there exists an
essential ideal $\hat{I}$ of $V$ such that ${\rm
Id}_{\tkk(V)}(\hat{I})$ is contained in $I$.
\end{lem}

\begin{proof}
Consider an essential ideal $I$ of $\tkk (V)$, which is a strongly
nondegenerate Lie algebra (by \cite[Proposition 2.6]{esther});  in
particular, it is semiprime. Therefore, we may apply Lemma
\ref{gradedideal} (i) and (ii) to find an essential graded ideal
${\hat I}_{-1}\oplus {\hat I}_0\oplus {\hat I}_1$ of $\tkk(V)$
contained in ${I}$. It can be shown, as in the proof of Theorem
\ref{equiv Jordan pair and Lie quotients}, that ${\hat
I}_{-1}\oplus (\,[{\hat I}_1,\,V^-]+[V^+,\,{\hat I}_{-1}]\,)
\oplus {\hat I}_1\subseteq { I}$ is an essential ideal of
$\tkk(V)$ and, by means of Lemma \ref{tkk filters}, ${\hat
I}:=({\hat I}_1, {\hat I}_{-1})$ is an essential ideal of $V$.
\end{proof}

For a strongly nondegenerate Jordan pair $V$, denote its maximal
Jordan pair of $\mathfrak{M}$-quotients by $Q_m(V)$ (see  \cite{gg}
for its construction).

\begin{thm} \label{Q_m(V) and Q_m(TKK(V))}
Assume that $\frac{1}{6}\in \Phi$.

\begin{enumerate}
\item[\rm{(i)}] Let $V$ be a strongly nondegenerate Jordan pair. Then
$$Q_m(V)=\bigg(\big(Q_m(\tkk(V))\big)_1,\,\big(Q_m(\tkk(V))\big)_{-1}\bigg)$$
is the maximal Jordan pair of $\,\mathfrak{M}$-quotients of V.
\item[\rm{(ii)}] If $L=L_{-1}\oplus L_0\oplus L_1$ is a strongly
non-degenerate\index{strongly!non-degenerate!Lie algebra} Jordan
3-graded Lie algebra\index{Jordan!3-graded Lie algebra} satisfying
that $Q_m(L)$ is Jordan 3-graded, then $$Q_m(L)\cong
Q_m(\tkk(V))\cong \tkk(Q_m(V)),$$ where $V=(L_1,\,L_{-1})$ is the
associated Jordan pair of $L$.
\end{enumerate}
\end{thm}

\begin{proof}
(i). The Lie algebras $Q_{\F_\tkk}(\tkk (V))$ and $Q_m(\tkk (V))$
are isomorphic by Lemmas \ref{tkk filters} and \ref{sturdy filter
of TKK} (see \cite{gg} for the definition of $Q_{\F_\tkk}(\tkk
(V)$). On the other hand, Theorem \ref{isom between Q_m and
Q_{gr-m}} (i) implies  that they are isomorphic to
$Q_{gr-m}(\tkk(V))$. (Note that $\tkk (V)$ is a strongly
nondegenerate Lie algebra by \cite[Proposition 2.6]{gg} so, it has
sense to consider its maximal graded algebra of quotients). Now,
the result follows by \cite[Theorem 3.2]{gg}.

(ii). The Lie algebra $L$ has zero center\index{center} because it
is strongly nondegenerate, hence $L\cong \tkk (V)$ (use Lemma
\ref{neher93}) and, obviously, $Q_m(L)\cong Q_m(\tkk (V))$. This one
is a strongly non-degenerate Lie algebra (by \ref{semiprimeness}) and has a 3-grading (by (i)) with associated Jordan
pair $Q_m(V)$. The hypothesis on $Q_m(L)$ allows us to use again
Lemma \ref{neher93} obtaining $Q_m(L) \cong \tkk(Q_m(V)).$
\end{proof}

The following is an example of a strongly non-degenerate Jordan
3-graded Lie algebra $L$ such that its maximal (graded) algebra of
quotients $Q_m(L)$ is not Jordan 3-graded. If we denote by $V$ the
associated Jordan pair of $L$, we obtain that $\tkk(Q_m(V))$ is not
(graded) isomorphic to $Q_m(L))$ (since $\tkk(Q_m(V))$ is Jordan
3-graded) which thereby means that the condition on $L$ in Theorem
\ref{Q_m(V) and Q_m(TKK(V))} (ii) is necessary.

\begin{ejm}
{\rm Denote by $\M_{\infty}(\R)=\cup^\infty_{n=1} \M_n(\R)$ the
algebra of infinite matrices with a finite number of nonzero entries
and consider
$$L:=\mathfrak{sl}_\infty(\R)=\{x\in \M_\infty(\R) \mid {\rm
tr}(x)=0\},$$ which is a simple Lie algebra of countable dimension
(see \cite[Theorem 1.4]{Baranov}).

Denote by $e_{ij}$ the matrix whose entries are all zero except for
the one in row $i$ and column $j$ and consider the orthogonal
idempotents $e:=e_{11}$ and $f:=\mbox{diag}(0,1,1,\ldots)$ (note
that $f\notin \M_\infty(\R)$); we can see $L$ as a 3-graded Lie
algebra by doing $L=L_{-1}\oplus L_0\oplus L_1,$ where $L_{-1}=eLf$,
$L_0=\{exe+fxf\mid x\in L\}$ and $L_1=fLe$.

Let $exe+\,fxf$ be an element of $L_0$ with $x=(x_{ij})\in \M_n(\R)$
for some $n\in \N$. Taking into account that ${\rm tr}(x)=0$, we
obtain:
$$exe+\,fxf= \sum^n_{i=2}\sum^n_{j=2}\,[-e_{1j},\,x_{ij}e_{i1}]\in [L_{-1},\,L_1],$$
This shows that $L_0=[L_{-1},\,L_1]$, i.e., $L$ is Jordan 3-graded.

In what follows, we will prove that $\der (L)$ is not Jordan
3-graded. The simplicity of $L$ implies that $Q_m(L)\cong \der (L)$;
on the other hand, the strongly nondegeneracy of $L$ allows us to
apply \cite[Proposition 1.7]{gg} obtaining that $\der (L)$ is
3-graded. Now, take, $\delta:=\Ad e$; one can easily check that:
$$\delta (L_{-1})\subseteq L_{-1}, \, \delta(L_0)=0 \, \mbox{ and }\,
\delta(L_1)\subseteq L_1,$$ which means that $\delta\in\der (L)_0$.
But note that $\delta \notin [\der(L)_{-1}, \,\der(L)_1]$ since the
elements of $[\der(L)_{-1}, \,\der(L)_1]$ have zero trace on every
finite dimensional subspace of $L$ while the trace of $\delta$ is
always nonzero. Therefore, $[\der(L)_{-1},\,\der(L)_1]\varsubsetneq
\der(L)_0$, i.e., $\der(L)$ is not Jordan 3-graded. }
\end{ejm}

\begin{rem}
{\rm Note that there exist non-trivial Jordan 3-graded Lie algebras
such that their maximal (graded) algebra of quotients are also
Jordan 3-graded Lie algebras. For example:

Let $F$ be a field and consider the Lie algebra
$$L:=\mathfrak{sl}_2(F)=\{x\in \M_2(F) \mid {\rm tr}(x)=0\}.$$ We have that $L$
is a Jordan 3-graded Lie algebra with the grading $L=L_{-1}\oplus
L_0\oplus L_1$, where
$$L_{-1}=Fe_{21},\,L_0=F(e_{11}-e_{22})\, \mbox{ and }\,
L_1=Fe_{12}.$$ Moreover, $L$ is a finite dimensional semisimple Lie
algebra and applying \cite[Lemma 3.9]{msm} we obtain that $L\cong
Q_m(L)$.}
\end{rem}

We finish the paper with analogues to
Theorem \ref{Q_m(V) and Q_m(TKK(V))}  but for Jordan triple
systems and Jordan algebras. In order to not enlarge the paper
 we refer the reader to \cite{loos, neher87, Mey} for
basic definitions and results on Jordan triple systems and algebras.

\begin{thm} \label{Q_m(T)}
Let $T$ be a strongly nondegenerate Jordan triple system over a ring
of scalars $\Phi$ containing $\frac{1}{6}$. Then the maximal Jordan
triple system of $\mathfrak{M}$-quotients of $T$ is the first
component of the maximal algebra of quotients of the TKK-algebra of
the double Jordan pair $V(T)=(T,\,T)$ associated to $T$, i.e.,
$$Q_m(T)=(Q_m(\tkk (V(T))))_1.$$
\end{thm}

\begin{proof}
The Jordan pair $V(T)=(T,\,T)$ is strongly nondegenerate since $T$
is so. By Theorem \ref{Q_m(V) and Q_m(TKK(V))} (i) we have
$$Q_m(V(T))=\bigg(\big(Q_m(\tkk(V(T)))\big)_1,\,\big(Q_m(\tkk(V(T)))\big)_{-1}\bigg).$$
The conclusion follows now from the versions of Lemmas \ref{tkk
filters} and \ref{sturdy filter of TKK} for Jordan triple systems
and from \cite[4.5 and Theorem 4.6]{gg}.
\end{proof}

\begin{thm}
Let $J$ be a strongly nondegenerate Jordan algebra over a ring of
scalars $\Phi$ containing $\frac{1}{6}$. Then
$$Q_m(J)=Q_m(J_T)=(Q_m(\tkk(V(J_T))))_1,$$ is the maximal Jordan
algebra of quotients of $J$, where $J_T$ denotes the Jordan triple
system associated to $J$ and $V(J_T)=(J_T,\,J_T)$ is the double
Jordan pair associated to $J_T$.
\end{thm}

\begin{proof}
Note that $J_T$ is a strongly nondegenerate Jordan triple system
by the strong nondegeneracy of the Jordan algebra $J$. From
\cite[5.4 and Theorem 5.5]{gg} it follows that the maximal Jordan
algebra of quotients $Q_m(J)$ is $Q_m(J_T)$. Finally apply Theorem
\ref{Q_m(T)} to reach the conclusion.
\end{proof}


\section*{acknowledgments}

The authors would like to thank Antonio Fern\'andez L\'opez for
his useful comments.

The first author has been supported by a grant by the Junta de
Andaluc\'\i a (BOJA n. 120 de 21 de junio de 2004), and both
authors have been supported by the Spanish MEC and Fondos FEDER
jointly, through project MTM2004-06580-C02-02, and by the Junta de
Andaluc\'{\i}a (projects FQM-336, FQM-1215 and FQM2467).


\end{document}